\documentclass{amsart}

\usepackage{amsthm} 
\usepackage{color}

\usepackage{graphicx,epsfig}
\usepackage{epstopdf}
\usepackage{amsmath, amssymb, latexsym, euscript,amscd}
\usepackage{url}
\usepackage[all]{xy}
\usepackage{psfrag}
\usepackage{mathrsfs}

\newcounter{notes}%

\definecolor{darkgreen}{rgb}{0.0, 0.5, 0.0}

\newtheorem{theorem}{Theorem}[section]

\newtheorem{corollary}[theorem]{Corollary}

\def\gap{\vspace{.3cm}\noindent}

\def\smallskip{\vspace{.15cm}}
\def\medskip{\vspace{.3cm}}
\def\text{\mbox}
\def\rh2{{\mathbb R}{\mathbb H}^2}
\def\ch2{{\mathbb C}{\mathbb H}^2}
\def\RR{{\mathbb R}}

\def\RP{{\mathbb{RP}}}

\def\SL{\operatorname{SL}}

\def\PO{\operatorname{PO}}

\def\H2R{{\mathbb H}^2\times {\mathbb R}}

\def\C2{\operatorname{C^2}}

\definecolor{back}{RGB}{255,255,255}
\definecolor{fore}{RGB}{0,0,0}
\definecolor{title}{RGB}{255,0,90}

\definecolor{green}{rgb}{0.0, 0.5, 0.0}
\definecolor{purple}{rgb}{0.5, 0.0, 0.5}
\definecolor{bluegreen}{rgb}{0.0,0.5, 0.5}
\definecolor{orange}{rgb}{1,0.5, 0.1}
\definecolor{redgreen}{rgb}{0.5, 0.5, 0.0}

\def\blue{\color{blue}}
\def\red{\color{red}}
\def\green{\color{green}}

\def\blue{\color{blue}}
\def\red{\color{red}}
\def\green{\color{green}}

\def\g2{{\green 2}}

\newcommand{\bv}{\left[\begin{array}{c}}
\newcommand{\ev}{\end{array}\right]}
\newcommand{\bbmat}{\begin{bmatrix}} 
\newcommand{\ebmat}{\end{bmatrix}}
\newcommand{\bmat}{\begin{matrix}} 
\newcommand{\emat}{\end{matrix}}
\newcommand{\bpmat}{\begin{pmatrix}} 
\newcommand{\epmat}{\end{pmatrix}}

\title{The Heisenberg Group acts on a strictly convex domain.}
\author{Daryl Cooper}
\date{\today}
\thanks{{The author acknowledges support from U.S. National Science Foundation grants DMS 1107452, 1107263, 1107367 ``RNMS: GEometric structures And Representation varieties'' (the GEAR Network). \\}
Cooper was partially supported by NSF grants DMS 1065939, 1207068 and  1045292\\
MSC 57N16, 57M50
}

\address{Department of Mathematics, University of California, Santa Barbara, CA 93106, USA}
\address{}

\email[]{cooper@math.ucsb.edu}

\begin{document}
\maketitle

Every linear group acts by isometries on some properly convex domain in real projective space. This follows
from the fact that action of $\SL(n,\RR)$ on the space of quadratic form in $n$ variables preserves the
properly convex cone consisting of positive definite forms. If $\Gamma$ is the holonomy of a properly
convex orbifold of finite volume then every virtually nilpotent group is virtually abelian, moreover every
unipotent element is conjugate into $\PO(n,1)$. A reference for all this is \cite{CLT1}.
 This paper gives the first example of a unipotent group
that is not virtually abelian and preserves a strictly convex domain. It answers a question asked
by Misha Kapovich.

The {\em Heisenberg group} is the subgroup  $H\subset \SL(3,\RR)$ of unipotent upper-triangular matrices. Define $\theta:H\rightarrow \SL(10,\RR)$ and $G=\theta(H)$ where
$$
\theta\bpmat
1 & a & c\\
0 & 1 & b\\
0 & 0 & 1\epmat
=\bpmat
1  &  2 a  &  2 c  &  a  &  a^2/2  &  a^3/6  &  b  &  2 a^2 + b^2/2  &  b^3/6 + 2 a c  &  
  {\red (a^4 + b^4)/24 + c^2}\\ 
  0  &  1  &  b  &  0  &  0  &  0  &  0  &  2 a  &  a b + c  &  b c\\
 0  &  0  &  1  &  0  &  0  &  0  &  0  &  0  &  a  &  c\\
 0  &  0  &  0  &  1  &  a  &  a^2/2  &  0  &  0  &  0  &  a^3/6\\
 0  &  0  &  0  &  0  &  1  &  a  &  0  &  0  &  0  &  a^2/2\\
 0  &  0  &  0  &  0  &  0  &  1  &  0  &  0  &  0  &  a\\
 0  &  0  &  0  &  0  &  0  &  0  &  1  &  b  &  b^2/2  &  b^3/6\\
 0  &  0  &  0  &  0  &  0  &  0  &  0  &  1  &  b  &  b^2/2\\
 0  &  0  &  0  &  0  &  0  &  0  &  0  &  0  &  1  &  b\\
 0  &  0  &  0  &  0  &  0  &  0  &  0  &  0  &  0  &  1\\
 \epmat
 $$
 It is clear that $\theta$ is injective and easy to check that it is a homomorphism.
 Since the center of $H$ is $Z\cong \RR$ and $H/Z\cong\RR^2$  it is also easy to check that every non-trivial element of $G$  
 has a unique largest Jordan block, and that this block has odd size. It easily follows that each element of $G$
 preserves some properly convex domain depending on that element, cf the discussion of parabolics in (2.9) of \cite{CLT1}.
  
 \begin{theorem} There is a strictly convex domain $\Omega\subset \RP^9$ that is preserved by $G$. This is an
 effective action of the Heisenberg group on $\Omega$ by parabolic isometries that are unipotent.
 \end{theorem}
 \begin{proof} The group $G$ acts affinely on the affine patch $[x_1:x_2:x_3:x_4:x_5:x_6:x_7:x_8:x_9:1]$ that
 we identify with $\RR^9$. Let $p\in \RR^9$ be the origin. Then $G\cdot p$ is
 $$( {\red (a^4 + b^4)/24 + c^2},bc,c,a^3/6,a^2/2,a,b^3/6,b^2/2,b)$$
 This orbit is an algebraic embedding $\RR^3\hookrightarrow\RR^9$ which limits on the single point 
 $$q=[1:0:0:0:0:0:0:0:0:0]\in\RP^9$$ in the hyperplane at infinity, $P_{\infty}$.
 This follows from the fact that $\red (a^4 + b^4)/24 + c^2$ dominates
  all the other entries whenever at least one of $|a|,|b|,|c|$
 is large. 
 
 Let $S\subset\RR^9$ be this orbit. Choose 10 random points on $S\subset\RP^9$ and compute the determinant, 
 $d$, of
 the corresponding 10 vectors in $\RR^{10}$. Then $d\ne 0$ therefore the interior $\Omega^+\subset\RR^9$ of the convex hull of $S$ has dimension 9. 
 
 Moreover the closure $\Omega'$ of $\Omega^+$ in $\RP^9$ is disjoint from
  the closure of the affine hyperplane $x_1=-1$, hence $\Omega^+$ is properly convex. 
 Since $\Omega'\cap P_{\infty}=q$ and $G$ preserves $q$ and $P_{\infty}$ and $G$ is unipotent, 
 it follows from (5.8) in \cite {CLT1} that $G$ preserves
 some strictly convex domain $\Omega\subset\Omega'$.
 \end{proof}
 
\begin{corollary} There is a strictly convex real projective manifold $\Omega/\Gamma$ of dimension 9
with nilpotent fundamental group $\Gamma\cong\langle \alpha,\beta:[\alpha,[\alpha,\beta]],[\beta,[\alpha,\beta]]\rangle$ that is not virtually abelian. Moreover $\Gamma$ is unipotent. \end{corollary}
\begin{proof}If $\Gamma$ is a lattice in $G$ then $\Omega/\Gamma$ is a strictly convex manifold with unipotent holonomy and
$\Gamma$ is nilpotent but not virtually abelian.\end{proof}

\gap

The genesis of this example is as follows. The image of $H$ in $\SL(6,\RR)$ under the irreducible representation
$\SL(3,\RR)\rightarrow\SL(6,\RR)$ is 
$$\left(\begin{array}{cccccc}
1 &  2 a &  {\red a^2} &  2 c &  2 a c &  {\blue c^2}\\ 
0 &  1 &  a &  b &  a b + c &  b c\\ 
0 &  0 &  1 &   0 &  2 b &  {\red b^2}\\ 
0 &  0 &  0 &  1 &  a &  c\\ 
0 &  0 &  0 &  0 &  1 &  b\\ 
0 &  0 &  0 &  0 &  0 &  1\\
  \end{array}\right)
  $$
and   preserves the properly convex domain $Q\subset\RP^5$ that is the projectivization
of the space of positive definite quadratic forms on $\RR^3$.

  The boundary of the closure of $Q$ consists of
semi-definite forms and contains flats, so $Q$ is {\em not strictly convex}. Let $A,B,C\in \SL(6,\RR)$ be the  
elements corresponding to one of $a,b,c$ being $1$ and the others 0. Each of $A,B,C$ has a parabolic fixed point in
$\partial Q$ corresponding to a rank 1 quadratic form. Every point in $Q$ converges to this parabolic
fixed point under iteration by the given group element. The fixed point for $A$ and $B$ are distinct and  lie in a flat
in $\partial Q$. 

The idea is to increase the dimension of the representation and use the extra dimensions
to add parabolic blocks of size $5$ onto $A$ (row 1 and rows 7-10) and onto $B$ (row 1 and rows 11-14) 
that commute and the parabolic fixed point of each block is the
rank-1 form that is a fixed point of $C$. This gives a 14-dimensional representation of $H$:

$$\left(\begin{array}{cccccccccccccc}
1 &  2 a &  {\red a^2} &  2 c &  2 a c &  {\blue c^2} &  a &  a^2/2 &  a^3/6 &  {\blue a^4/24} &  b &  b^2/2 &   b^3/6 &  {\blue b^4/24}\\
 0 &  1 &  a &  b &  a b + c &  b c &  0 &  0 &  0 &  0 &  0 &  0 &  0 &  0\\
 0 &  0 &  1 &  0 &  2 b &  {\red b^2} &  0 &  0 &  0 &  0 &  0 &  0 &  0 &  0\\
 0 &  0 &  0 &  1 &  a &  c &  0 &  0 &  0 &  0 &  0 &  0 &  0 &  0\\
 0 &  0 &  0 &  0 &  1 &  b &  0 &  0 &  0 &  0 &  0 &  0 &  0 &  0\\
 0 &  0 &  0 &  0 &  0 &  1 &  0 &  0 &  0 &  0 &  0 &  0 &  0 &  0\\
 0 &  0 &  0 &  0 &  0 &  0 &  1 &  a &  a^2/2 &  a^3/6 &  0 &  0 &  0 &  0\\
 0 &  0 &  0 &  0 &  0 &  0 &  0 &  1 &  a &  a^2/2 &  0 &  0 &  0 &  0\\
 0 &  0 &  0 &  0 &  0 &  0 &  0 &  0 &  1 &  a &  0 &  0 &  0 &  0\\
 0 &  0 &  0 &  0 &  0 &  0 &  0 &  0 &  0 &  1 &  0 &  0 &  0 &  0\\
 0 &  0 &  0 &  0 &  0 &  0 &  0 &  0 &  0 &  0 &  1 &  b &  b^2/2 &  b^3/6\\
 0 &  0 &  0 &  0 &  0 &  0 &  0 &  0 &  0 &  0 &  0 &  1 &  b &  b^2/2\\
 0 &  0 &  0 &  0 &  0 &  0 &  0 &  0 &  0 &  0 &  0 &  0 &  1 &  b\\
 0 &  0 &  0 &  0 &  0 &  0 &  0 &  0 &  0 &  0 &  0 &  0 &  0 &  1\\
\end{array}\right)$$

The top-left $6\times 6$ block is the image of $H$ in $\SL(6,\RR)$.
The entries in $A^n$ and $B^n$ {\red grow like $n^2$}. This is beaten
by the growth of some entries in the added blocks of size $5$ which {\blue grow like $n^4$}. This gives rise to a representation
of $H$ of dimension $6+4+4=14$. The orbit of 
$$[0: 0: 0: 0: 0: 1: 0: 0: 0: 1: 0: 0: 0: 1]$$
is
$$[{\blue(a^4 + b^4)/24 + c^2}: 
 b c: b^2: c: b: 1: a^3/6: a^2/2: a: 1: b^3/6: b^2/2: b: 1]$$
 so there is a codimension-4 projective hyperplane that is preserved, and which is defined by
$$x_6=x_{10}=x_{14}\qquad x_5=x_{13}\qquad x_3=2x_{12}$$ The
restriction to this hyperplane gives $\theta$.

\end{document}